\documentclass[fontsize=14pt,DIV=1]{scrartcl}

\usepackage{amsmath,amsfonts,amssymb,amstext,amsthm,bbm,graphicx,stackengine,tikz,stmaryrd,marvosym}

\usepackage[cal=boondoxo]{mathalfa} 
\usepackage{newpxtext,newpxmath}

\newtheorem{thm}{Theorem}
\newtheorem{cor}{Corollary}[thm]

\newtheorem*{thm*}{Theorem}
\newtheorem*{lemma*}{Lemma}
\newtheorem*{prop*}{Proposition}
\newtheorem*{rmk*}{Remark}
\newtheorem*{cor*}{Corollary}
\newtheorem*{claim*}{Claim}

\newtheoremstyle{named}{}{}{\itshape}{}{\bfseries}{.}{.5em}{#1 \thesubsection}
\theoremstyle{named}

\newtheorem*{namedprop}{Proposition}
\newtheorem*{namedcor}{Corollary}

\newtheoremstyle{quoted}{}{}{\itshape}{}{\bfseries}{.}{.5em}{#1 #3}
\theoremstyle{quoted}
\newtheorem*{quotedtheorem}{Theorem}
\newtheorem*{quotedprop}{Proposition}
\newtheorem*{quotedcor}{Corollary}

\numberwithin{figure}{section}

\let\OLDthebibliography\thebibliography
\renewcommand\thebibliography[1]{
  \OLDthebibliography{#1}
  \small
  \setlength{\parskip}{3pt}
  \setlength{\itemsep}{3pt plus 0.3ex}
}

\newcommand{\mypar}[1]{\refstepcounter{subsection}\paragraph{\thesubsection}\label{#1}}
\newcommand{\myspar}[2]{\refstepcounter{subsubsection}\subparagraph{\thesubsubsection.}\label{#1}{\em #2.}}

\def\tr{\operatorname{tr}}

\def\logT{{\frac{\log T}{2\pi}}}

\begin{document}

\title{{\Large On the critical points of random matrix characteristic polynomials and of the Riemann $\xi$-function}}
\author{\normalsize Sasha Sodin\textsuperscript{1}}
\date{{\small\today}}
\maketitle

\footnotetext[1] {\,\,\,School of Mathematical
Sciences, Queen Mary University of London, London E1~4NS, United Kingdom $\binampersand$ School of Mathematical
Sciences, Tel Aviv University, Tel Aviv, 69978, Israel.  E-mail:
a.sodin{\MVAt}qmul.ac.uk. Supported in part by the European
Research Council start-up grant 639305 (SPECTRUM).}

\begin{abstract} A one-parameter family of point processes describing the 
distribution of the 
critical points of the characteristic polynomial of large random Hermitian matrices on 
the scale of mean spacing is investigated. Conditionally on the Riemann hypothesis 
and the multiple correlation conjecture, we show that one of these limiting processes also 
describes the distribution of the critical points of the Riemann $\xi$-function on 
the critical line.

We prove that each of these processes boasts stronger level repulsion than the sine process  describing the limiting statistics of the eigenvalues: the probability to find $k$
critical points in a short interval is comparable to the probability to
find $k+1$ eigenvalues there. 
We also prove a similar property for the critical
points and zeros of the Riemann $\xi$-function, conditionally on the Riemann hypothesis
but not on the multiple correlation conjecture.
\end{abstract}

\section{Introduction}

\mypar{intr1}
Let $\mathfrak{Si}$ be the sine point process, i.e.\ a random locally finite 
subset  of $\mathbb R$ the distribution of which is determined by
\begin{equation}\label{eq:defsine} \mathbb{E} \!\!\!\!\sum_{\substack{x_1,  \cdots , x_k \in \mathfrak{Si}\\\text{pairwise distinct}}}\!\!\!\!
f(x_1, \cdots, x_k) =  \int d^k x f(x_1, \cdots, x_k) \det\left( \frac{\sin \pi(x_j - x_l)}{\pi (x_j - x_l)}\right)_{j,l=1}^k\end{equation}
The sine process describes the eigenvalue distribution of  random Hermitian matrices 
on the scale of mean eigenvalue spacing. For complex Wigner matrices (a class of
high-dimensional Hermitian random 
matrices with independent entries above the main diagonal, cf.\ Section~\ref{intr}), this is expressed by the following relation, 
which is part of a series of results obtained by Erd\H{o}s--Yau, Tao--Vu and coworkers, 
see \cite{EY,TV}: if $H_N$ is a sequence of   random matrices of growing dimension satisfying the assumptions listed in Section~\ref{intr}, and $\lambda_{j,N}$ are the eigenvalues
of $H_N/\sqrt{N}$, then for $E \in (-2, 2)$
\begin{equation}\label{eq:universality}\left\{ (\lambda_{j,N} - E) \frac{N\sqrt{4-E^2}}{2\pi} \right\}_{j=1}^N \longrightarrow \mathfrak{Si} \quad
\text{in distribution} \end{equation}
with respect to the topology defined by continuous test functions of compact support.
The factor $ \frac{2\pi}{N\sqrt{4-E^2}}$ by which the eigenvalues are scaled is the
(approximate) mean spacing between eigenvalues near $E$. Similar results are 
available for other random matrix ensembles, see the monographs \cite{AGZ09,PShch} and references therein.

The  correlation conjecture of Montgomery \cite{Mont} in the extended version of
 Rudnick--Sarnak \cite{RS} and Bogomolny--Keating \cite{BK1,BK2} states that a similar relation
 holds for the zeros of the Riemann $\zeta$-function on the critical line: if $t$ is
 chosen uniformly at random in $[0, T]$, then  
\begin{equation}\label{eq:multcor} \left\{ (\gamma - t) \frac{\log T}{2\pi} \,\, \Big| \,\, \zeta(\frac{1}{2} + i \gamma) = 0\right\} \overset{??}\longrightarrow \mathfrak{Si} \quad
\text{in distribution} \end{equation}
(the question marks are put to emphasise that this relation is still conjectural).
The scaling factor $\frac{2\pi}{\log T}$ is the approximate mean spacing between
the zeros with imaginary part near $T$.
The results of \cite{Mont,Hej,RS} imply that, conditionally on the Riemann hypothesis,
convergence holds for a restricted family of test functions.

\medskip\noindent
Following Aizenman and Warzel \cite{AW} and Chhaibi, Najnudel and Nikeghbali \cite{CNN}, consider the random entire function 
\begin{equation}\label{eq:defphi}\Phi(z) = \lim_{R \to \infty} \prod_{x \in \mathfrak{Si} \cap (-R, R)} (1 - z/x)~. \end{equation}
We study the one-parameter family of  point processes
\[ \mathfrak{Si}'_a = \left\{ z \in \mathbb{C}\, \Big| \, \Phi'(z) = a \Phi(z)\right\} \subset \mathbb{R}~, \quad a \in \mathbb{R}~. \]
For any $a$, the points of $\mathfrak{Si}'_a$ interlace with those of $\mathfrak{Si}$. Therefore the statistical properties of $\mathfrak{Si}'_a$ on long scales are
very close to those of $\mathfrak{Si}$. Also, $\mathfrak{Si}'_a \to \mathfrak{Si}$
as $a \to \infty$. 

On the other hand, on short scales the
processes $\mathfrak{Si}'_a$ are much more rigid. To quantify this, introduce the
 events
\[ \Omega_k(\mathfrak{Si}, \epsilon) = \left\{ \# \left[\mathfrak{Si} \cap (-\epsilon, \epsilon)\right]  \geq k\right\}~, \quad
\Omega_k(\mathfrak{Si}'_a, \epsilon) = \left\{ \# \left[ \mathfrak{Si}'_a \cap (-\epsilon, \epsilon)\right] \geq k \right\}~. \]
From the special case 
\[ \mathbb{E} \frac{(\# [\mathfrak{Si} \cap (-\epsilon, \epsilon)])!}{(\# [\mathfrak{Si} \cap (-\epsilon, \epsilon)] - k)!} = \int_{(-\epsilon,\epsilon)^k} d^k x\,\det \left(\frac{\sin \pi(x_j - x_l)}{\pi (x_j - x_l)}\right)_{j,l=1}^k~\]
of (\ref{eq:defsine}), the sine process boasts the following  repulsion property:
for any $k \geq 1$,
\begin{equation}\label{eq:rep0}
\mathbb{P} (\Omega_k (\mathfrak{Si}, \epsilon)) 
= \mathcal c_k \epsilon^{k^2} + \mathcal o(\epsilon^{k^2})~, \quad \epsilon \to + 0~,
\quad \text{where $0 < c_k < \infty$.}\end{equation}
For comparison, the probability of the corresponding event in the standard
Poisson process decays as $\epsilon^k$. Our first result is 
\begin{thm}\label{thm:sine}
For any $a \in \mathbb R$, $k \geq 2$ and $0 < \epsilon < 1$
\begin{equation}\label{eq:thmsine}
\mathbb{P} \left(\Omega_k(\mathfrak{Si}'_a, \epsilon) \setminus \Omega_{k+1}(\mathfrak{Si}, 5 \epsilon)\right) \leq C_k \left(\epsilon \log \frac{1}{\epsilon}\right)^{(k+2)^2}~.\end{equation}
\end{thm}
\noindent That is, $k$-tuples of critical points in a short interval (for a fixed value
of $k$) are mostly due to
$(k+1)$-tuples of zeros in a slightly larger interval. From (\ref{eq:thmsine}) and (\ref{eq:rep0}) we  have
\[ (\mathcal c_{k+1} + \mathcal o(1))\epsilon^{(k+1)^2} \leq \mathbb{P} (\Omega_k (\mathfrak{Si}'_a, \epsilon))
\leq  (5^{(k+1)^2}\mathcal c_{k+1} + \mathcal o(1)) \epsilon^{(k+1)^2}~, \quad \epsilon \to +0~.\]
A slightly more careful argument shows that
\begin{cor}\label{cor:limit} For any $k \geq 2$ there exists a limit
\[\mathcal c_k' = \lim_{\epsilon\to +0} \frac{\mathbb{P} (\Omega_k(\mathfrak{Si}'_a, \epsilon))}
{\epsilon^{(k+1)^2}} \in [\mathcal c_{k+1}, (1 + \frac{4}{k-1})^{(k+1)^2} \mathcal c_{k+1}]~,\]
independent of $a \in \mathbb R$.
\end{cor}

Combining Corollary~\ref{cor:limit} with a  result of Aizenman--Warzel \cite{AW} which
is stated as Proposition~\ref{aw} below, we obtain
\begin{cor}\label{cor:wig} Let $(H_N)$ be a sequence of complex Wigner matrices satisfying the assumptions listed in Section~\ref{intr}, and let $(\lambda_{j,N}')_{j=1}^{N-1}$ be the critical points of the characteristic polynomial
$P_N(\lambda) = \det(H_N / \sqrt{N} - \lambda)$. For $E \in (-2,2)$ and $k \geq 2$,
\[ \lim\limits_{N \to \infty}\mathbb{P} \left\{ \# \left[ \left| \lambda'_{j,N} - E \right| < \frac{2 \pi \epsilon}{N\sqrt{4-E^2}} \right]\geq k \right\}
= (\mathcal c_k' + \mathcal o(1))\epsilon^{(k+1)^2}~. \]
\end{cor}
The stronger repulsion between the critical points can be seen on Figures~\ref{fig} and
\ref{fig2}.
\begin{figure}[h]
\centerline{\includegraphics[trim={0cm 0cm  0cm 0}, scale=.48,angle=270]{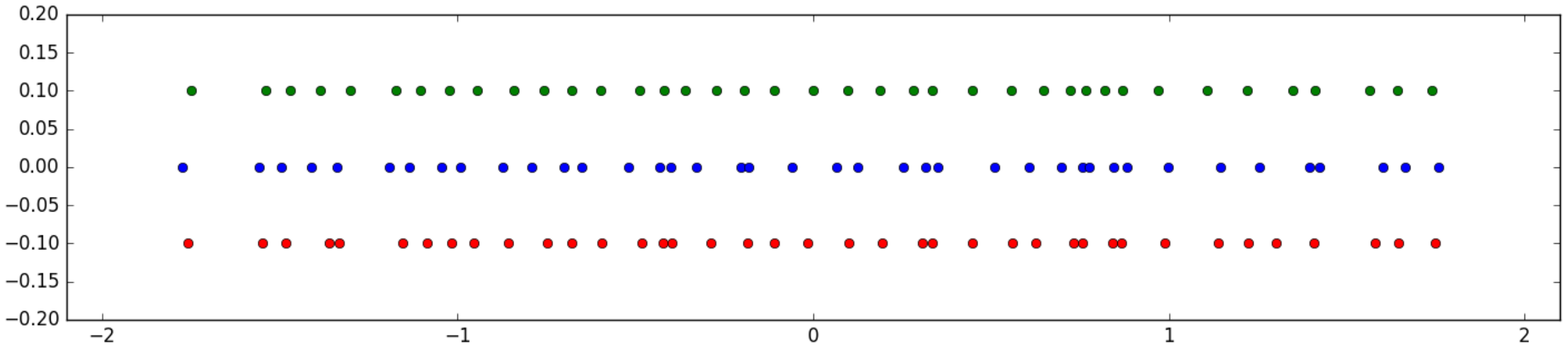}}
\caption{The critical points of the characteristic polynomial  of GUE$_{40}$, the eigenvalues,
and the eigenvalues of a principal submatrix of dimension $39$.\label{fig}}
\end{figure}

\begin{figure}[h]
\centerline{\includegraphics[trim={0cm 0cm  0cm 0}, scale=.4,angle=0]{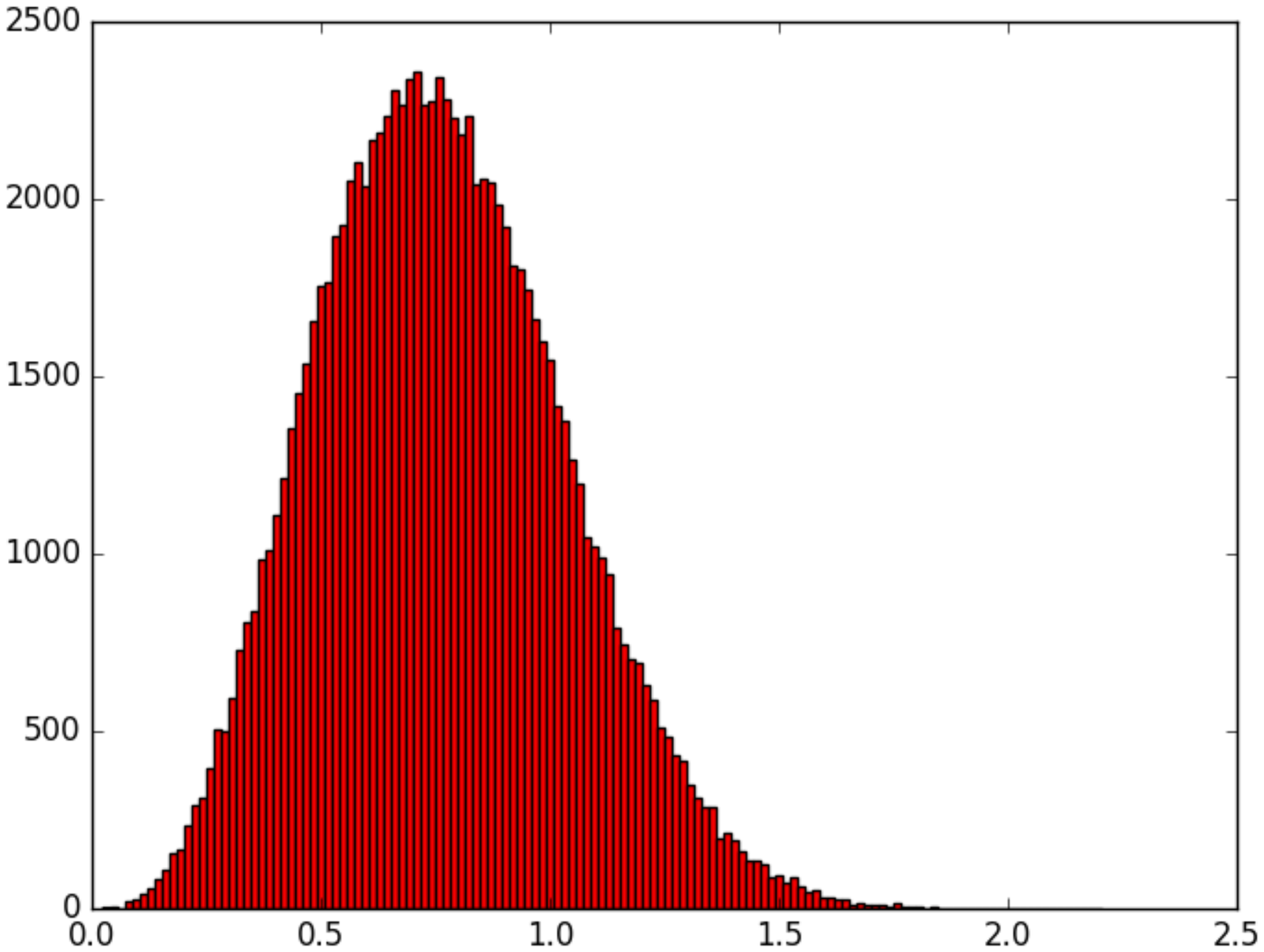}
\includegraphics[trim={0cm 0cm  0cm 0}, scale=.4,angle=0]{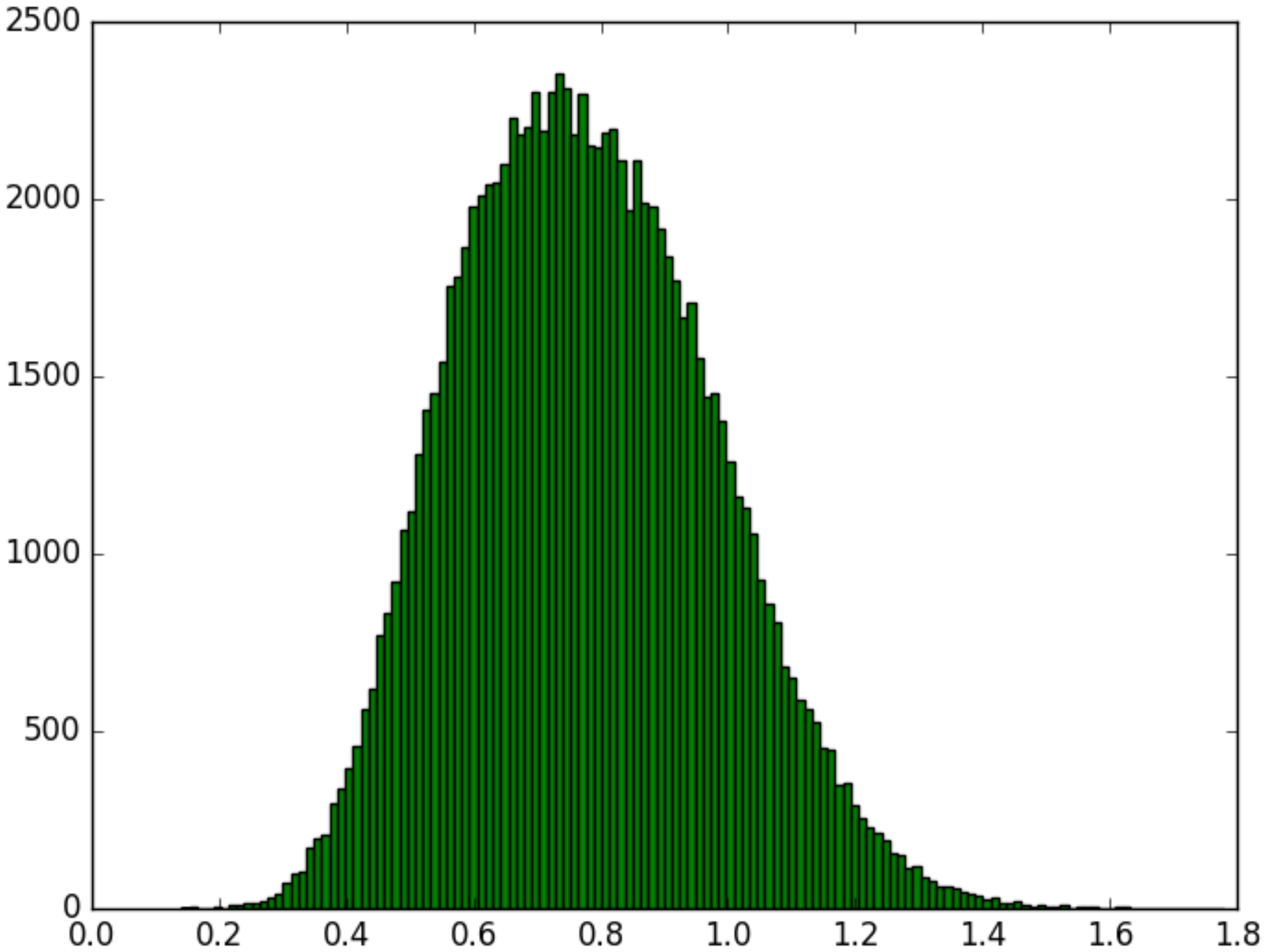}}
\caption{\label{fig2}A histogram of the absolute values of the eigenvalue (red, left) and
critical value (green, right) second closest to zero for GUE$_{50}$, multiplied by $50 \pi$}
\end{figure}

\medskip\noindent In the number-theoretic setting, we consider the Riemann $\xi$-function
\[ \xi (s)={\tfrac  {s(s-1)}{2\pi^{s/2}}}\Gamma \left({\tfrac  {s}{2}}\right)\zeta (s)~. \]
This is an entire function which is real on the critical line; its zeros 
coincide with the non-trivial zeros of the $\zeta$-function. Conditionally
on the Riemann hypothesis, the zeros  of $\xi'$ lie on the critical line $\Re s = \frac12$ and interlace with the zeros of $\xi$ (cf.\ Section~\ref{xi}). Assuming the Riemann
hypothesis together with the multiple correlation conjecture, we prove (see Corollary~\ref{xi}) that
\begin{equation}\label{eq:multcor'''} \left\{ (\gamma' - t) \frac{\log T}{2\pi} \,\, \Big| \,\, \xi'(\frac{1}{2} + i \gamma') = 0\right\} \overset{??}\longrightarrow \mathfrak{Si}'_0 \quad
\text{in distribution.} \end{equation}
See Figure~\ref{fig:comp}.
\begin{figure}[h]
\centerline{\includegraphics[trim={0cm 0cm  0cm 0}, scale=.48,angle=0]{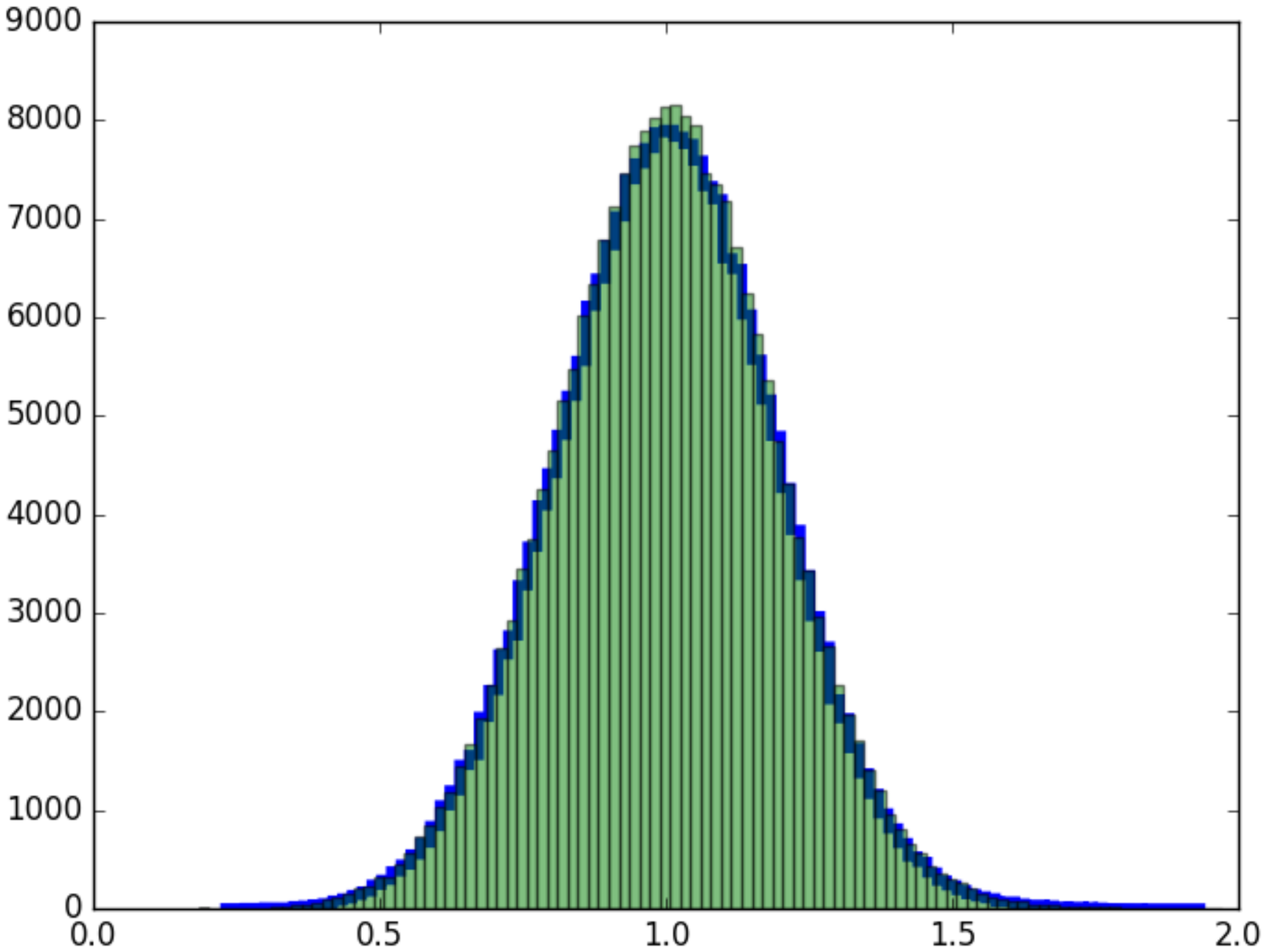}}
\caption{A histogram of spacings between the critical points for GUE$_{300}$ near $E=0$ 
(blue line) and for the $\xi$-function near $\frac12 + 2^{32}i$ (green bars). Data courtesy of Dave Platt.\label{fig:comp}}
\end{figure}

Denote 
\[\begin{split}
\Omega_k(\xi, T, \epsilon) &= 
\left\{ 0 \leq t \leq T \, \Big| \, \# \left[ \gamma \in (t - \frac{2\pi\epsilon}{\log T}, 
t + \frac{2\pi\epsilon}{\log T})~, \,\, \xi(\frac12 + i\gamma) = 0\right] \geq k \right\}\\
 \Omega_k(\xi', T, \epsilon) &= 
\left\{ 0 \leq t \leq T \, \Big| \, \# \left[ \gamma' \in (t - \frac{2\pi\epsilon}{\log T}, 
t + \frac{2\pi\epsilon}{\log T})~, \,\, \xi'(\frac12 + i\gamma') = 0\right] \geq k \right\}
\end{split}\]
From Corollary~\ref{cor:limit} and (\ref{eq:multcor'''}), we obtain 
\begin{cor}\label{cor:xi} Assume the Riemann hypothesis and the multiple correlation conjecture 
(\ref{eq:multcor}). Then
\[ \lim\limits_{T \to \infty} \frac1T \operatorname{mes}(\Omega_k(\xi', T, \epsilon))= (\mathcal c_{k}' + \mathcal o(1)) \epsilon^{(k+1)^2}~. \]
\end{cor}
\noindent We also prove the following less conditional result with a similar message:
$k$-tuples of critical points of the $\xi$-function crowding short intervals are mostly
a consequence of $(k+1)$-tuples of zeros crowding slightly larger intervals.
\begin{thm}\label{thm:xi}
Assume the Riemann hypothesis. For any $k \geq 2$, $0 < \epsilon < 1$, $R \geq 5$
\[ \frac1T \operatorname{mes} \left(\Omega_k(\xi', T, \epsilon) \setminus \left(\Omega_{k+1}(\xi, T, 5\epsilon) \cup \Omega_{k+2}(\xi, T, R \epsilon)\right) \right)\leq \frac{C}{e^{ckR}}~.  \]
\end{thm}

\mypar{mot1} 
Let us discuss the motivation for these results. The traditional object
 of study in random matrix theory is the joint distribution
of the eigenvalues. 
The local eigenvalue statistics, i.e.\ the study of eigenvalues
on the scale of mean eigenvalue spacing, is of particular interest due to the
robust (universal) nature of the limiting objects.

Recently, the value distribution of the characteristic polynomial of a random
matrix also received significant attention. While the characteristic polynomial
is determined by the eigenvalues, its restriction to an interval depends both
on the eigenvalues inside the interval and those outside it. Therefore the
statistical properties of the characteristic polynomial on the scale of mean 
eigenvalue spacing are not necessarily determined by the local 
eigenvalue statistics.

As one varies the argument over an interval containing many eigenvalues for a given realisation of the random matrix, the value
 of the polynomial shows huge variations by the orders of magnitude. We refer to
 the works \cite{FB,FK14} for a discussion and references, in particular, for a statistical 
 mechanics  perspective on  the absolute value of the characteristic polynomial as a disordered landscape (a Boltzmann weight 
with a log-correlated potential).  
It was found that characteristic polynomials of random
matrices can be  used to model the value distribution of the Riemann
zeta function on the critical line \cite{KS00,HKOC01,GHK}.
 In particular, the statistics related to the global maximum of the modulus of characteristic polynomial has been studied for matrices drawn from the circular ensemble,
 and these results were used to study (on the physical and mathematical
 levels of rigour) the properties of the global maximum of
 $|\zeta(1/2+it)|$ in various intervals \cite{FHK12,FK14,ABH,ABB,PZ,CNT,CNN,Naj,ABBRS}. 
 Parallel questions for the characteristic polynomial of Hermitian
 random matrices were investigated  in \cite{FS16,FL16,FKS}.
 
\medskip\noindent
The sequence of positions of the local maxima (and minima) of the characteristic
polynomial which we study in this paper is  one of the natural characteristics of 
the random landscape. 

\medskip\noindent Second, the zeros and the critical points of the characteristic polynomial form an interlacing pair of sequences. As put forth by Kerov \cite{Kerov_book,Kerov_interl}, such pairs naturally appear in numerous
problems of analysis, probability theory, and representation theory, and their
limiting properties in various asymptotics regimes are of particular importance.

In the recent work \cite{s_interl} we studied the statistical properties of the
zeros and the critical points from the point
of view of the global regime, namely, the fluctuations of linear statistics. In particular,
the fluctuations differ from those of another natural interlacing pair, formed
by the eigenvalues of a random matrix and those of a principal submatrix;
see Erd\H{o}s and Schr\"oder \cite{ES}.

Here we study the joint distribution of the zeros and the critical points in the
local regime, i.e.\ on the scale of the mean spacing. See further Corollary~\ref{aw} and Paragraph~\ref{r.comp}, and also Figure~\ref{fig}.

\medskip\noindent
Finally, the strong repulsion between the critical points is an instance of a general phenomenon: the zeros of the derivative
of a polynomial  with real zeros (or of an analytic function in the Laguerre--P\'olya
class) are more evenly spaced than the zeros of the
original polynomial. In the deterministic setting, this phenomenon goes back
to the work of  Stoyanoff  and M.\ Riesz \cite{St}, see \cite{FR} for a historical discussion.
Theorem~\ref{thm:sine} and its corollaries provide additional examples in
the random setting.

Here we remark that, under repeated differentiation, the zeros become more and more rigid and approach
an arithmetic progression (after proper rescaling and with the right order of limits).
This was established in various settings in \cite{PS,FY,Ki} (of which \cite{Ki}
is applicable to the $\xi$-function), and is probably true for $\Phi(z)$ of (\ref{eq:defphi})
as well.

\mypar{mot2}
The second source of motivation comes from number theory (disclaimer:  there are no new number-theoretic ingredients 
in our arguments). 

The study of the zeros of $\xi'$
goes back to the work of Levinson \cite{Lev2} and Conrey \cite{Con1,Con2}, who
obtained unconditional lower bounds on the fraction of zeros  lying on the critical line  (see \cite{Sel,Lev1}, the more
recent \cite{Feng} and references therein for the corresponding results pertaining to
the zeros of $\xi$).

More recently,
Farmer, Gonek and Lee \cite{FGL} and further Bian \cite{Bian} and Bui \cite{Bui} 
studied the correlations between the zeros of $\xi'$, arguing that a  detailed understanding of the
joint statistics of the zeros of $\xi$ and $\xi'$  may allow to rule out the so-called
Alternative Hypothesis, according to which the spacings between
nearest high-lying zeros of $\xi$ are close to half-integer multiples of the mean
spacing. The main result in \cite{FGL}
asserts that, conditionally on the Riemann  hypothesis,
\begin{equation}\label{eq:fgl}  \frac{1}{N(T)}\!\!\!
\sum_{\substack{0 < \gamma', \tilde\gamma' < T \\ \xi'(\frac12 + i\gamma') = \xi'(\frac12 + i\tilde\gamma') = 0}} \!\!\! \frac{4e^{i\alpha(\gamma'-\tilde\gamma') \log T}}{4 + (\gamma' - \tilde\gamma')^2} \to
|\alpha| - 4|\alpha|^2 + \sum_{k=1}^\infty \frac{(k-1)!}{(2k)!} (2|\alpha|)^{2k+1}
\end{equation}
as $T \to \infty$, for $0 < \alpha < 1$, where
\begin{equation}\label{eq:riem}N(T) = \# \left\{ 0 \leq t \leq T \, \mid \, \xi(1/2 + it) = 0 \right\}  = \frac{T}{2\pi} \log T \, (1 + \mathcal o(1))~.\end{equation}
For comparison,  Montgomery showed \cite{Mont} that (conditionally on the Riemann hypothesis)
for $0 < \alpha < 1$
\begin{equation}\label{eq:mont} 
\frac{1}{N(T)}
\sum_{\substack{0 < \gamma, \tilde\gamma < T \\ \xi(\frac12 + i\gamma) = \xi(\frac12 + i\tilde\gamma) = 0}} \frac{4e^{i\alpha(\gamma-\tilde\gamma) \log T}}{4 + (\gamma - \tilde\gamma)^2} \to
|\alpha|~, \quad T \to \infty~.
\end{equation}
The $\alpha \to 0$ asymptotics of form factors on the left-hand side of (\ref{eq:fgl}) 
and (\ref{eq:mont}) capture the behaviour of the spacings between zeros on long scales. The short scale behaviour roughly corresponds to the $\alpha \to \infty$ asymptotics
of the form factor. 

The pair correlation conjecture of Montgomery \cite{Mont} states that for $\alpha \geq 1$ the limit
of the left-hand side of (\ref{eq:mont}) is equal to $1$, similarly to the form
factor of the sine process. Further, Hejhal \cite{Hej}
and  Rudnick and Sarnak \cite{RS} extended the result (\ref{eq:mont}) to higher
correlations, which, together with the arguments of Bogomolny--Keating \cite{BK1,BK2}, led to the conjecture \cite{BK1,BK2,RS} that the full 
asymptotic distribution of the zeros of $\xi$ on the scale of mean spacing
is described by the sine process of random matrix theory, (\ref{eq:multcor}). 

On the other hand, a conjectural 
description of the full limiting distribution of the critical points seems to have been 
missing. Our results (Corollary~\ref{xi} below) provide such a conjectural
description: it turns out that (\ref{eq:multcor}) formally implies that
\begin{equation}\label{eq:multcor'} \left\{ (\gamma' - t) \frac{\log T}{2\pi} \,\, \Big| \,\, \xi'(\frac{1}{2} + i \gamma') = 0\right\} \overset{??}\longrightarrow \mathfrak{Si}'_0 \quad
\text{in distribution.} \end{equation}
Then, Corollary~\ref{cor:xi} provides a conditional description of the left tail
of the spacing distribution, whereas Theorem~\ref{thm:xi} provides some less
conditional information. 

\medskip\noindent
Here we also mention the works  devoted to the zeros of $\zeta'$, particularly,
\cite{LM,Mez1,DFFHMP,Lester} (and references therein). The zeros of $\zeta'$ do not lie on the critical line, and are therefore
thematically more distant from the current study; their counterparts in  random matrix setting (in a sense made precise in the aforementioned works) are the
critical points of the characteristic polynomial of a circular ensemble.

\mypar{tight} Let us briefly discuss the derivation of Corollaries~\ref{cor:wig} and \ref{cor:xi}
from Theorem~\ref{thm:sine}; see Sections~\ref{aw} and \ref{xi} for the precise
definition and proofs. 
The collection of all the critical points of the characteristic polynomial is determined by
the collection of all the eigenvalues. However, it is not a priori clear whether this relation
persists in the local limit regime: that is, whether the conditional distribution of the critical points in an interval of length, say, $(5 \times \text{mean spacing})$, conditioned
on the eigenvalues outside a concentric interval of length
$(R \times \text{mean spacing})$ degenerates in the limit $R \to \infty$, uniformly
in the matrix size. 

Technically, (\ref{eq:universality}) amounts to the convergence in distribution of linear
statistics of the form
\begin{equation}\label{eq:linC0}\sum_{j=1}^N f\left((\lambda_{j,N} - E) \frac{N\sqrt{4-E^2}}{2\pi} \right)~, \quad  
f \in C_0(\mathbb R)\end{equation}
to the corresponding statistics of the sine process. It 
is possible to extend this to other integrable test functions satisfying mild
regularity conditions. On the other hand, the critical points are controlled by
linear statistics corresponding to functions of the form $f(\lambda) = \frac{1}{\lambda - z}$,
$z \in \mathbb C \setminus \mathbb R$; the asymptotics of such linear 
statistics is not a formal consequence of (\ref{eq:universality}).

Recently, Aizenman and Warzel \cite{AW} put forth a general condition  which
ensures that (\ref{eq:universality}) can be upgraded to the convergence of such linear statistics. In the
setting of Wigner matrices, they verified the condition using the local semicircle law of Erd\H{o}s--Schlein--Yau \cite{ESY} and the universality results of Erd\H{o}s--Yau, Tao--Vu and coworkers \cite{EY,TV}, and obtained:
\begin{quotedprop}[\ref{aw}' (Aizenman--Warzel)]
Let $P_N(z) = \det(H_N/\sqrt{N} - z)$, where $H_N$ is a sequence of complex Wigner matrices satisfying the assumptions listed in Section~\ref{intr}. Then for $E \in (-2, 2)$
\begin{align}\label{eq:awintr}
\frac{2\pi}{N \sqrt{4-E^2}} \frac{P_N'\left(E + \frac{2\pi z}{N \sqrt{4-E^2}}\right)}{P_N\left(E + \frac{2\pi z}{N \sqrt{4-E^2}}\right)} &\overset{\text{distr}}\longrightarrow \frac{\Phi'(z)}{\Phi(z)} + \frac{\pi E z}{\sqrt{4 - E^2}}~, &N \to \infty \\\label{eq:awintr2}
\frac{P_N\left(E + \frac{2\pi z}{N \sqrt{4-E^2}}\right)}{P_N(E)}&\overset{\text{distr}}\longrightarrow \Phi(z) \exp\left[  \frac{\pi E z}{\sqrt{4 - E^2}}\right]~,  &N \to \infty
\end{align}
with respect to the topology of locally uniform convergence on $\mathbb C \setminus \mathbb R$ (in the first relation) and of locally uniform convergence on $\mathbb C$
(in the second relation).
\end{quotedprop}
The second relation follows from the first one by (careful) integration.\footnote{some care
is required to upgrade  uniform convergence on compact subsets of $\mathbb C \setminus \mathbb R$ to uniform convergence  on compact subsets of $\mathbb C$ not containing the poles of the limiting meromorphic function.}.
We mention that Chhaibi, Najnudel and Nikeghbali established a counterpart of (\ref{eq:awintr})--(\ref{eq:awintr2}) for the Circular Unitary Ensemble;
see further Paragraph~\ref{r.oth}.

\medskip\noindent
We show that a similar result holds for the Riemann $\xi$-function, conditionally on the Riemann
hypothesis; see Proposition~\ref{xi}. Somewhat similar
statements have been proved for the  $\zeta$-function, cf.\ \cite{FGLL,GGM,Rodgers}. 
Here we quote the following corollary (the non-trivial statement is $\Longrightarrow$):
\begin{quotedcor}[\ref{xi}'] Conditionally on the Riemann hypothesis, the multiple correlation conjecture (\ref{eq:multcor}) is equivalent to each of the
following relations: 
\begin{align}
\frac{2\pi i}{\log T}\frac{\xi'(\frac12 + i(t + \frac{2\pi z}{\log T}))}{\xi(\frac12 + i(t + \frac{2\pi z}{\log T}))} &\overset{\text{distr}}\longrightarrow \frac{\Phi'(z)}{\Phi(z)}~,  &T \to \infty \\
\label{eq:xiconvintr} 
\frac{\xi(\frac12 + i(t + \frac{2\pi z}{\log T}))}{\xi(\frac12 + it)} &\overset{\text{distr}}\longrightarrow \Phi(z)~,  &T \to \infty
\end{align}
when $t$ is chosen uniformly at random in $[0, T]$.
\end{quotedcor}

The combination of Theorem~\ref{thm:sine} with Proposition~\ref{aw} and Corollary~\ref{xi} implies Corollaries~\ref{cor:wig} and \ref{cor:xi}. Roughly speaking, the 
asymptotics of the statistics which depend on ratios of
the characteristic polynomial [or the $\xi$-function], for example, the joint distribution of the zeros of the first
$k$ derivatives, are determined by the sine process.

Here we remark that the fluctuations of linear statistics corresponding to functions
such as $f(\lambda) = \log(\lambda - z)$  contain a component that depends on the eigenvalues [or zeros] outside the microscopic window. This component is
not universal, and, for the $\xi$-function, contains an arithmetic piece; see Gonek, Hughes and Keating \cite{GHK} and references therein. Therefore the limiting value
distribution of the characteristic polynomial differs from that of the $\xi$-function
(for any deterministic regularisation).

\section{Convergence of random analytic functions}
\mypar{intr}
Let $H_N = (H(i,j))_{i,j=1}^N$ be a complex Wigner matrix, which for us is a Hermitian random matrix such that
\begin{enumerate}
\renewcommand{\theenumi}{\alph{enumi}}
\item $\left\{ (\Re H(i,j))_{i < j}, (\Im H(i,j))_{i < j}, (H(i,i))_i\right\}$ are independent random variables;
\item $ (\Re H(i,j))_{i < j}, (\Im H(i,j))_{i < j}$ are identically distributed, and
\[ \mathbb{E} H(i,j)=0~, \,\,\mathbb{E} |H(i,j)|^2 = 1~, \,\, \mathbb{E} \exp(\delta |H(i,j)|^2) < \infty \]
for some $\delta > 0$;
\item $(H(i,i))_i$ are identically distributed, and
\[ \mathbb{E} H(i,i)=0~, \quad \mathbb{E} H(i,i)^2  < \infty~, \quad  \mathbb{E} \exp(\delta |H(i,i)|^2) < \infty~. \]
\end{enumerate}
The main example is the Gaussian Unitary Ensemble (GUE), in which the joint probability density of the matrix elements of $H$ is proportional to $\exp(- \frac12 \tr H^2)$.

\medskip\noindent 
Denote by $(\lambda_{j,N})_{j=1}^N$ the eigenvalues of $H/\sqrt{N}$.
The global statistics of
the eigenvalues is described by Wigner's law:
\begin{equation}\label{eq:wig}
\frac{1}{N} \sum_{j=1}^N \delta(E- \lambda_{j,N}) {\longrightarrow} \rho(E) dE
\end{equation}
weakly in distribution, where $\rho(E) =  \frac{1}{2\pi} \sqrt{(4-E^2)_+}$
is the semicircular density. Due to the interlacing between the critical points $\lambda_{j,N}'$ and the zeros $\lambda_{j,N}$ of $P_N(\lambda) = \det (H_N / \sqrt N - \lambda)$, one also has:
\[ \frac{1}{N-1} \sum_{j=1}^{N-1} \delta(E- \lambda'_{j,N}) {\longrightarrow} \rho(E) dE~.\]

\medskip\noindent
The local  statistics of $\lambda_{j,N}$ are described by
the sine point process $\mathfrak{Si}$ defined in (\ref{eq:defsine}) 
(see e.g.\ \cite{Sosh} for general properties of determinantal point processes):  for any $E \in (-2,2)$ one has the convergence in distribution:
\begin{equation}\label{eq:tosine}
\sum_j \delta\left(u - (\lambda_{j,N} - E) N \rho(E)\right) \longrightarrow \sum_{x \in \mathfrak{Si}} \delta(u-x)~, \quad N \to \infty~.
\end{equation}
This result was proved in the 1960-s for the Gaussian Unitary Ensemble (the eigenvalues
of which form a determinantal point process); see \cite{Mehta}.  In a series of
works by Erd\H{o}s--Yau, Tao--Vu, and coworkers, (\ref{eq:tosine}) was  generalised to  Wigner matrices satisfying assumptions such as a.--c.\ above; see \ \cite{EY,TV} and references therein.

The correlation conjecture of Montgomery \cite{Mont} in the extended version of
Rudnick--Sarnak \cite{RS} and Bogo\-mol\-ny-Keating \cite{BK1,BK2}
asserts that a  similar statement holds for the non-trivial  zeros of the $\zeta$-function:
\begin{equation}\label{eq:tosinezeta}
\sum_{\xi(\frac12 + i\gamma)=0}\delta\left(u - (\gamma - t) \frac{\log T}{2\pi}\right) \overset{??}\longrightarrow \sum_{x \in \mathfrak{Si}} \delta(u-x)~, \quad T \to \infty
\end{equation}
in distribution, when $t$ is uniformly chosen from $[0, T]$.

\medskip\noindent
The relations (\ref{eq:tosine}) and (\ref{eq:tosinezeta}) mean that
\begin{align}\label{eq:mean1} \sum_j f\left(\lambda_{j,N} - E) N \rho(E)\right) \longrightarrow \sum_{x \in \mathfrak{Si}} f(x)\\
\label{eq:mean2} \sum_{\xi(\frac12 + i\gamma)=0}f\left((\gamma - t) \frac{\log T}{2\pi}\right) \overset{??}\longrightarrow \sum_{x \in \mathfrak{Si}} f(x)\end{align}
in distribution for continuous test functions $f$ of compact support. It is possible to 
extend this to integrable test functions satisfying mild regularity conditions. 
On the other hand, going beyond integrable
functions requires additional information about the zeros lying far
away from the microscopic window. It turns out that the second
relation, (\ref{eq:mean2}) is (conditionally) valid for test functions  $f(x) = 1/(x-z)$, if the sums are properly
regularised, whereas the first relation, (\ref{eq:mean1}) requires a deterministic correction depending on
$E$; see Sections~\ref{xi} and \ref{aw} (relying on the works \cite{GGM} and \cite{AW}, respectively). These properties imply that the critical points depend
quasi-locally, so to speak, on the zeros / eigenvalues.

\mypar{aw}
We recall the construction of Aizenman--Warzel \cite{AW}.  Recall that a function $w : \mathbb C \setminus \mathbb R \to \mathbb C$ belongs to the Nevanlinna [= Herglotz = Pick] class ($w \in \mathcal R$) if it is analytic and
\[ \overline{w(z)} = w(\bar z)~, \quad \frac{\Im w(z)}{\Im z} > 0~. \]
The class $\mathcal R$ is equipped with the topology of pointwise convergence
on compact subsets of $\mathbb C \setminus \mathbb R$.

\smallskip\noindent
Denote
\begin{equation}\label{eq:w}
W (z) = \lim_{R \to \infty} \sum_{x \in \mathfrak{Si}} \left[ \frac{1}{x-z} - \frac{1}{x-iR} \right] + i\pi =
\lim_{R \to \infty} \sum_{x \in \mathfrak{Si} \cap (-R,R)} \frac{1}{x - z}~.
\end{equation}
The two limits exist and coincide according to a general criterion of \cite{AW}, and $W(z)$
is a random element of the Nevanlinna class. Also note that $-W(z)$ is the logarithmic
derivative of the function $\Phi(z)$ from (\ref{eq:defphi}).

As before, let $P_N(z) = \det(H_N/\sqrt{N} - z)$ be the characteristic polynomial of $H_N/\sqrt{N}$, and denote
\[ W_N(z; E) =  - \frac{1}{N \rho(E)} \frac{p_N'\left(E + \frac{z}{N\rho(E)}\right)}{p_N\left(E + \frac{z}{N\rho(E)}\right)} = \sum_j \frac{1}{(\lambda_{j,N}-E)N\rho(E)-z}~. \]
\begin{quotedprop}[\ref{aw} (Aizenman--Warzel)]
For $|E|<2$,
\begin{align}\label{eq:convW}
W_N(z; E) \,\,&\overset{\operatorname{distr}}\longrightarrow \,\,W(z) - \frac{\pi E}{\sqrt{4-E^2}}\\
\frac{P_N\left(E + \frac{2\pi z}{N \sqrt{4-E^2}}\right)}{P_N(E)} \,\,&\overset{\operatorname{distr}}\longrightarrow \,\,\Phi(z) \exp\left[\frac{\pi E z}{\sqrt{4 - E^2}}\right]~.\end{align}
\end{quotedprop}
\begin{proof}
The first statement is proved in  \cite[Corollary 6.5]{AW}, their argument relies on the results obtained in the works \cite{ESY,EY,TV} on
the local eigenvalue statistics of Wigner matrices, and on the general
theory of random Nevanlinna functions which was developed in \cite{AW}.
The second statement  follows from the first one by (carefully)
integrating from $0$ to $z$.
\end{proof}

Denote by $w^{-1}(a)$ the collection of solutions of $w(z) = a$. Observe that the map $w \mapsto w^{-1}(a)$ from $\mathcal R \cap \{\text{meromorphic functions}\}$ to locally
finite \mbox{(multi-)}subsets of $\mathbb R$  is continuous. From this observation and (\ref{eq:convW}) we deduce:
\begin{namedcor}
 Let $(H_N)$ be a sequence of random matrices satisfying the assumptions listed in Section~\ref{intr}. Then for any $E \in (-2,2)$
\begin{equation}\label{eq:conv}\begin{split}
&\left( \sum_j \delta\left(u - (\lambda_{j,N} - E) N \rho(E)\right)~, \sum_j \delta\left(u - (\lambda_{j,N}' - E) N \rho(E)\right)\right) \\
&\quad\longrightarrow  \left(W^{-1}(\infty), \, W^{-1}(-a)\right)
= \left( \mathfrak{Si}, \mathfrak{Si}'_{a}\right)
 \end{split}\end{equation}
 in distribution, where $a = -\frac{\pi E}{\sqrt{4-E^2}}$.
\end{namedcor}

To recapitulate, the non-obvious part of the statement is  that the zeros $\lambda_{j,N}$ which
are not in a $\mathcal O(1/N)$--neighbourhood of $E$ influence the critical points $\lambda_{j,N}'$ near $E$ only via the deterministic quantity $\pi E (4-E^2)^{-1/2}$.
Colloquially, the conditional distribution of the critical points in $(E - r/N, E+r/N)$ given
the eigenvalues in $(E - R/N, E+R/N)$ degenerates in the limit $R \to \infty$.

\begin{proof}[Proof of Corollary~\ref{cor:wig}] Follows from Theorem~\ref{thm:sine}
and Corollary~\thesubsection.
\end{proof}

\mypar{xi} Denote
\[ W_{t,T}(z) = - \frac{2\pi i}{\log T} \frac{\xi'(\frac12 + i [t + \frac{2\pi z }{\log T}])}{\xi(\frac12 + i [t + \frac{2\pi z}{\log T}])}~, \quad 0 \leq t \leq T~. \]
Assuming the Riemann hypothesis, $W_{t,T}$ belongs to the Nevanlinna class. \footnote{the property $W_{t,T} \in \mathcal{R}$ is
independent of $t$ and $T$, and is in fact equivalent to the Riemann hypothesis.}
We shall treat $t$ as a random variable uniformly chosen in $[0, T]$, and denote
the corresponding random function by $W_T$. The next proposition
is close to the  results for the $\zeta$-function which were proved in \cite{GGM} and \cite{FGLL}.

\begin{namedprop} Assume the Riemann hypothesis.
Let $T_n \to \infty$ be a sequence and $\mathfrak P$ --- a
point process such that, for $t$ uniformly chosen in $[0, T_n]$,
\[ \sum_{\xi(\frac12 + i\gamma) = 0} \delta(u - (\gamma - t)\frac{\log T_n}{2\pi})\to \mathfrak P \quad \text{in distribution}~,
\quad n \to \infty~.\]
Then $W_{T_n} \to W_{\mathfrak P}$ in distribution,
where $W_{\mathfrak P}$ is the unique random Nevanlinna  function such
that
\[ W_{\mathfrak P}^{-1}(\infty) = \mathfrak P~, \quad  W_{\mathfrak P}(i\infty) = i\pi~,\]
and in particular, 
\[ \left(\sum_{\xi(\frac12 + i\gamma) = 0} \delta(u - (\gamma - t)\frac{\log T_n}{2\pi}), \sum_{\xi'(\frac12 + i\gamma) = 0} \delta(u - (\gamma' - t)\frac{\log T_n}{2\pi}) \right)\to \left(\mathfrak P,  \mathfrak P'_0 \right)~,\]
where $\mathfrak P'_0 =  W_{\mathfrak P}^{-1}(0)$.
\end{namedprop}

\begin{rmk*} Assuming the Riemann hypothesis, it follows from the results of  Fujii \cite{Fuj,Fuj2} that
\begin{equation}\label{eq:fujii}
\frac{1}{T} \int_0^T \left(N(t + \frac{2\pi R}{\log T}) - N(t) - R\right)^2 dt  \leq C \log(e+R)~,	 
\end{equation}
therefore the family of point processes $W_{T}^{-1}(\infty)$ is precompact (i.e.\ for any
finite interval $I$ the family of random variables $\# (W_{T}^{-1}(\infty) \cap I)$ is  tight),
and, moreover, any limit point $\mathfrak P$ satisfies:
\[ \mathbb{E} \# \left[ \mathfrak P \cap [-R, R] \right] = 2R~, \quad
\mathbb{E} \left| \# \left[ \mathfrak P \cap [0, R] \right]  - R\right|^2 \leq C\log(e+R)~. \]
In particular, the conditions of \cite[Theorem 4.1]{AW} are satisfied, and therefore
the function $W_{\mathfrak P}$ (uniquely) exists.
\end{rmk*}

\noindent Proposition~{\thesubsection} implies
\begin{namedcor} Assume the Riemann hypothesis and the multiple correlation
conjecture (\ref{eq:multcor}). Then
\[ W_{T} \overset{\text{distr}}{\longrightarrow} W~, \quad  \frac{\xi(\frac12 + i(t + \frac{2\pi z}{\log T}))}{\xi(\frac12 + it)} \overset{\text{distr}}{\longrightarrow} \Phi(z)~, \quad T \to \infty \]
and consequently
\begin{align*}
&\left( \sum_{\xi(\frac12 + i\gamma) = 0} \delta(u - (\gamma - t)\frac{\log T}{2\pi}),
\sum_{\xi'(\frac12 + i\gamma') = 0} \delta(u - (\gamma' - t)\frac{\log T}{2\pi})\right)\overset{\text{distr}}{\longrightarrow} (\mathfrak{Si}, \mathfrak{Si}'_0)~.
\end{align*}
\end{namedcor}

\begin{proof}[Proof of Corollary~\ref{cor:xi}] Follows from Theorem~\ref{thm:sine}
and Corollary~\thesubsection.
\end{proof}

The proof of  Proposition~{\thesubsection} relies on the following lemma. Related results go back to the
work of Selberg \cite{Sel}. We essentially follow the argument in \cite{GGM},
relying on the work of Montgomery \cite{Mont}. 

\begin{lemma*} Assuming the Riemann hypothesis, one has for any $R > 0$:
\begin{align}
\label{eq:a}
&\lim_{T \to \infty} \int_0^T \frac{dt}{T} W_{t,T}(iR) = i\pi \\
\label{eq:b}
 &\limsup_{T \to \infty} \int_0^T\frac{dt}{T}  \,\,
|W_{t,T}(iR) - i\pi|^2 \leq \frac{C}{R^2}~.
\end{align}
\end{lemma*}

\begin{proof}[Proof of Proposition~\thesubsection]
The convergence of $W_{T}$ to $\Phi$ follows from the lemma, in combination with the general criterion of Aizenman and Warzel \cite[Theorem 6.1]{AW}; it implies the other two
statements.
\end{proof}

\begin{proof}[Proof of Lemma]
To prove (\ref{eq:a}), consider the integral of $\xi'/\xi$ along the closed contour
$\Gamma$ composed of the segments connecting the points
\[ \frac12 - \frac{R}{\log T}, \frac12 + \frac{R}{\log T},
\frac12 + \frac{R}{\log T} +T i,  \frac12 + \frac{R}{\log T} +T' i,
\frac12 - \frac{R}{\log T} +T i \]
counterclockwise (see Figure~\ref{fig:contour}), where $T'$ is the real number
closest to $T$ such that there are no zeros of the $\xi$-function in the
$1/(100 \log T)$-neighbourhood of $\frac12 + iT$. 
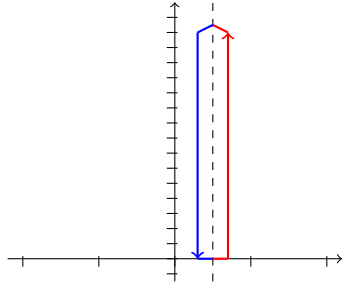
\begin{figure}[h]\label{fig:contour}
\centerline{
\begin{tikzpicture}
      \draw[->] (-2.2,0) -- (2.2,0) ;
      \draw[->] (0,-.3) -- (0,3.4);
       \draw[dashed] (1/2,-.3) -- (1/2,3.4);
      \foreach \x in {-2,...,2}
     		\draw (\x,1pt) -- (\x,-3pt);
	\foreach \y in {-.2,0,...,3.2}
     		\draw (1pt,\y) -- (-3pt,\y);
		\draw[blue,thick] (.3, 0) -- (.5, 0);
				\draw[red,thick] (.5, 0) -- (.7, 0);
        \draw[red,thick,->] (.7, 0) -- (.7, 3);
         \draw[red,thick] (.7, 3) -- (.5, 3.1);
          \draw[blue,thick] (.5, 3.1) -- (.3, 3);
         \draw[blue,thick,->] (.3,3) -- (.3, 0);
 \end{tikzpicture}}
 \caption{The contour $\Gamma$ from the proof of Lemma~\thesubsection}
\end{figure}
By the residue theorem and the asymptotics (\ref{eq:riem}) of $N(T)$, the integral is equal to
\[ 2 \pi i \, \# \left\{ \text{zeros of $\xi$ with imaginary part in $[0, T']$}\right\}
= 2\pi i \,\, \frac{T \log T}{2\pi} \, (1+\mathcal o(1))~.\]
On the other hand, from the functional equation $\xi(1-z) = \xi(z)$,
the integral along the left vertical line  is equal to the integral along the right vertical line, and
the integral along the bottom horizontal line is zero; the integral along the two segments on the top is 
negligible (as one can see, for example, from (\ref{eq:hadlogder}) below). Therefore 
\[ \int_0^T (\xi'/\xi)(\frac12 + i(t + iR/\log T)) dt = -\pi \,\, \frac{T \log T}{2\pi} \, (1+\mathcal o(1))~.\]
which is equivalent to (\ref{eq:a}). We note for the sequel that (\ref{eq:a}) implies
the following smoothened version:
\begin{equation}\label{eq:a.smooth} 
\lim_{T \to \infty} \int_{-\infty}^\infty \frac{T \, dt}{\pi(t^2 + T^2)} W_{t,T}(iR) = i\pi~.
\end{equation}

\noindent To prove (\ref{eq:b}), we use the Hadamard product representation (cf.\ \cite[2.12]{Titch})
\[\xi(\frac{1}{2} + iz) = \xi(\frac12) e^{\hat b z} \prod_{\xi(\frac12 + i\gamma)=0} (1 - z/\gamma) e^{z/\gamma}~,\]
which implies that
\begin{equation}\label{eq:hadlogder} W_{t,T}(iR) = \sum_{\gamma} \left[ \frac{1}{(\gamma-t) \logT - iR} -\frac{1}{ \gamma\,\logT } \right] - \frac{2\pi \hat b}{\log T}~.\end{equation}
Integrating with the weight $T/(\pi(t^2 + T^2))$ and using (\ref{eq:a.smooth}) and the 
Cauchy theorem, we
obtain:
\begin{equation}\label{eq:a.smooth'}
\lim_{T \to \infty}  \sum_\gamma \left[ \frac{1}{(\gamma-iT)\log T - iR} -\frac{1}{ \gamma\,\log T } \right] = \frac{i}{2}~.\end{equation}
Note that (\ref{eq:a.smooth'}) holds for any real $R$ (positive or negative).

Now we compute
\[I(R,T)
= \int_{-\infty}^\infty \frac{T dt}{\pi(t^2 + T^2)} \left| \frac{W_{t,T}(iR)}{2\pi} + \frac{\hat b}{\log T}\right|^2 = \sum_{\gamma,\tilde\gamma} I_{\gamma,\tilde\gamma}(R,T)~, \]
where 
\[\begin{split} &I_{\gamma,\tilde\gamma}(R, T)
= \int_{-\infty}^\infty \frac{T dt}{\pi(t^2 + T^2)}\left[ \frac{1}{(\gamma-t)\log T - iR} -\frac{1}{ \gamma\,\log T } \right] \\
&\qquad\times\left[ \frac{1}{(\tilde\gamma-t)\log T +iR} -\frac{1}{\tilde\gamma\,\log T } \right]~.\end{split}\]
By the Cauchy theorem,
\[ I_{\gamma,\tilde\gamma}(R, T) = I_{\gamma,\tilde\gamma}'(R,T) +  I_{\gamma,\tilde\gamma}''(R,T)~,\]
where
\[ \begin{split}
&I_{\gamma,\tilde\gamma}'(R,T) = \left[ \frac{1}{(\gamma-iT)\log T - iR} -\frac{1}{ \gamma\,\log T } \right]  \left[ \frac{1}{(\tilde\gamma-iT)\log T +iR} -\frac{1}{\tilde\gamma\,\log T } \right]\\
&I_{\gamma,\tilde\gamma}''(R,T) = \frac{-2iT}{T^2 + (\tilde\gamma+ \frac{iR}{\log T})^2} \, 
\left[ \frac{1}{(\gamma - \tilde\gamma)\log T - 2iR} -\frac{1}{ \gamma\,\log T }\right]\, \frac{1}{\log T}~.
\end{split}\]
In view of (\ref{eq:a.smooth'}),
\begin{equation}\label{eq:ingr1} \lim_{T \to \infty} \sum_{\gamma, \tilde\gamma}I_{\gamma,\tilde\gamma}'(R,T) = (i/2)^2 = -1/4~. \end{equation}
To estimate the sum of $I_{\gamma,\tilde\gamma}''(R,T)$, let 
\[\begin{split}
& J_{\gamma,\tilde\gamma}(R, T) = \frac{2T}{T^2 + (\tilde \gamma)^2} \frac{2R}{(\gamma - \tilde\gamma)^2\log^2 T +4R^2} \, \frac{1}{\log T}~, \\
&J_{\gamma,\tilde\gamma}'(R, T) = \Re I_{\gamma,\tilde\gamma}''(R,T)  - J_{\gamma,\tilde\gamma}(R, T)~. \end{split}\]
Using the estimates
\[\begin{split}
\left| \frac{1}{T^2 + (\tilde\gamma + \frac{iR}{\log T})^2} -  \frac{1}{T^2 + (\tilde\gamma)^2} \right|
&\leq \frac{3R}{(T^2 +(\tilde\gamma)^2)^2}\\
\left| \frac{1}{(\gamma - \tilde\gamma)\log T - 2iR} -\frac{1}{ \gamma\,\log T }\right| 
&\leq \frac{\sqrt 2(|\tilde\gamma| \log T + 2R)}{|\gamma| \log T(|\gamma - \tilde\gamma|\log T + 2R)}~,\end{split}\]
we deduce that
\begin{equation}\label{eq:ingr3} 
\lim_{T \to \infty} \sum_{\gamma, \tilde \gamma} J'_{\gamma,\tilde\gamma}(R, T)  = 0~.\end{equation} 
Now we turn to $\sum_{\gamma,\tilde\gamma} J_{\gamma,\tilde\gamma}(R, T)$ and show that 
\begin{equation}\label{eq:ingr2}
\limsup_{T \to \infty}  \left|\sum_{\gamma,\tilde\gamma}J_{\gamma,\tilde\gamma}(R, T) - \frac12 \right| \leq \frac{\operatorname{Const}}{R^2}~.
 \end{equation}
It will suffice to prove that
\begin{equation}\label{eq:ingr2a}
 \limsup_{T \to \infty} \left|\ \frac{2\pi}{T \log T} \sum_{0 \leq \gamma, \tilde \gamma \leq T} \frac{2R}{(\gamma - \tilde\gamma)^2 \log^2 T + 4 R^2} 
  -1\right| \leq \frac{\operatorname{Const}}{R^2}~.\end{equation}
Let 
\begin{equation}\label{eq:defF} F(\alpha, T) = \frac{2\pi}{T \log T} \sum_{0 \leq \gamma, \tilde\gamma \leq T} T^{i\alpha(\gamma - \gamma')} w(\gamma - \tilde\gamma)~, \quad w(u) = \frac{4}{4+u^2}~. \end{equation}
Montgomery showed \cite{Mont} that 
\begin{equation}\label{eq:mont1} F(\alpha, T) = |\alpha| + T^{-2|\alpha|} \log T (1 + \mathcal o(1)) + \mathcal o(1)~, \quad
|\alpha| < 1 \end{equation}
(which implies (\ref{eq:mont})). Therefore
\begin{equation}\label{eq:GG} \sup_x \int_x^{x+1} F(\alpha, T) d\alpha \leq C\end{equation}
(see \cite{Gold,GG}, where this is proved with $C = \frac83 + \epsilon$ and $\frac{29}{12}+\epsilon$, respectively). From the definition (\ref{eq:defF}) of $F$ we have (cf.\ \cite[(2.11)]{GGM}):
\begin{equation}\label{eq:withw}
 \frac{2\pi}{T \log T} \sum_{0 \leq \gamma, \tilde \gamma \leq T} \frac{2R}{(\gamma - \tilde\gamma)^2 \log^2 T + 4 R^2}  w(\gamma - \tilde\gamma) 
= \int_0^\infty F(\alpha, T) e^{-2R|\alpha|} d\alpha~,
\end{equation}
and thus, from (\ref{eq:mont1}) and (\ref{eq:GG}),
\begin{equation}\label{eq:withw1}
\limsup_{T \to \infty} \left| \frac{2\pi}{T \log T} \sum_{0 \leq \gamma, \tilde \gamma \leq T} \frac{4R}{(\gamma - \tilde\gamma)^2 \log^2 T + 4 R^2}  w(\gamma - \tilde\gamma) -1 \right| \leq \frac{C}{R^2}~. 
\end{equation}
Observing that 
\[ \lim_{T \to \infty} \frac{2\pi}{T \log T} \sum_{0 \leq \gamma, \tilde \gamma \leq T} \frac{4R}{(\gamma - \tilde\gamma)^2 \log^2 T + 4 R^2}  (1-w(\gamma - \tilde\gamma)) = 0~, \]
we obtain (\ref{eq:ingr2a}) and thus (\ref{eq:ingr2}). 
The relations  (\ref{eq:ingr1}), (\ref{eq:ingr3}) and (\ref{eq:ingr2}) imply that
\[ \lim_{R \to \infty} \limsup_{T \to \infty} |I(R, T) - 1/4| = 0 \]
and hence (\ref{eq:b}) holds.
\end{proof}

\section{Repulsion}\label{S:rep}
In this section we prove Theorems~\ref{thm:sine} and \ref{thm:xi} and 
Corollary~\ref{cor:limit}, the (re-)formulation of
which we recall for the convenience of the reader. Let
\[ \Omega_k(\mathfrak{Si}, \epsilon) = \left\{ \# \left[\mathfrak{Si} \cap (-\epsilon, \epsilon)\right]  \geq k\right\}~, \quad
\Omega_k(\mathfrak{Si}'_a, \epsilon) = \left\{ \# \left[ \mathfrak{Si}'_a \cap (-\epsilon, \epsilon)\right] \geq k \right\}~. \]
\begin{quotedtheorem}[\ref{thm:sine}']
For any $a \in \mathbb R$, $k \geq 2$, $0 < \epsilon < \frac{1}{8\max(|a|, 2e)}$ and $R \geq 5$
\[ \mathbb{P} \left(\Omega_k(\mathfrak{Si}'_a, \epsilon) \setminus \left( \Omega_{k+1}(\mathfrak{Si}_a, (1 + \frac{4}{k-1}) \epsilon) \cup \Omega_{k+2}(\mathfrak{Si}_a, R\epsilon)\right)\right) \leq  2 \exp(- \frac{kR}{64})~.\]
\end{quotedtheorem}
Taking $R = 1000 k \log\frac1\epsilon$ and using (\ref{eq:rep0}), we obtain the version
of Theorem~\ref{thm:sine} stated in the introduction. 
\begin{quotedcor}[\ref{cor:limit}] For any $k\geq 2$ there exists a limit 
\[\mathcal c_k' = \lim_{\epsilon\to +0} \frac{\mathbb{P} (\Omega_k(\mathfrak{Si}'_a, \epsilon))}
{\epsilon^{(k+1)^2}} \in (\mathcal c_{k+1}, (1 + \frac{4}{k-1}) \mathcal c_{k+1})~,\]
independent of $a \in \mathbb{R}$.
\end{quotedcor}

\medskip\noindent
Next, denote
\[\begin{split}
\Omega_k(\xi, T, \epsilon) &= 
\left\{ 0 \leq t \leq T \, \Big| \, \# \left[ \gamma \in (t - \frac{2\pi\epsilon}{\log T}, 
t + \frac{2\pi\epsilon}{\log T})~, \,\, \xi(\frac12 + i\gamma) = 0\right] \geq k \right\}\\
 \Omega_k(\xi', T, \epsilon) &= 
\left\{ 0 \leq t \leq T \, \Big| \, \# \left[ \gamma' \in (t - \frac{2\pi\epsilon}{\log T}, 
t + \frac{2\pi\epsilon}{\log T})~, \,\, \xi'(\frac12 + i\gamma') = 0\right] \geq k \right\}
\end{split}\]
\begin{quotedtheorem}[\ref{thm:xi}']
Assume the Riemann hypothesis. For any $k \geq 2$, $0 < \epsilon <1$, $R \geq 1 + \frac{4}{k-1}$
\[ \frac1T \operatorname{mes} \left(\Omega_k(\xi', T, \epsilon) \setminus \left(\Omega_{k+1}(\xi, T, (1 + \frac{4}{k-1})\epsilon) \cup \Omega_{k+2}(\xi, T, R\epsilon)\right) \right)\leq \frac{C}{e^{ckR}}~.  \]
\end{quotedtheorem}
\begin{rmk*}
The coefficient $1 + \frac{4}{k-1}$ in these results can be further improved to $1 + \frac2k+ \mathcal o(1)$. 
\end{rmk*}

\mypar{rep} 

\begin{proof}[Proof of Theorem~\ref{thm:sine}'] 
By Proposition~\ref{aw}, the theorem is equivalent to Corollary~\ref{cor:wig}, and,
moreover, to its special case pertaining to  one (arbitrary) ensemble of random matrices; we choose the Gaussian Unitary Ensemble. (We could equally work 
directly with $\mathfrak{Si}$; in that case we would need to regularise all the sums.)
Denote
\[ x_{j,N} = (\lambda_{j,N} - E)N\rho(E)~, \quad x'_{j,N} = (\lambda'_{j,N} - E)N\rho(E)~, \]
where $a = - \frac{\pi E}{\sqrt{4-E^2}}$. 
 Let us first show that for any $R \geq1 + 4/(k-1)$
\begin{equation}\label{eq:3opt}\begin{split}
& \Omega_k(\mathfrak{Si}'_a, \epsilon) \setminus \left( \Omega_{k+1}(\mathfrak{Si}_a, (1 + \frac{4}{k-1}) \epsilon) \cup \Omega_{k+2}(\mathfrak{Si}_a, R\epsilon)\right)
  \\&\subset \left\{  \sum_{|x_{j,N} - \epsilon| \geq (R-1)\epsilon} \frac{1}{x_{j,N}-\epsilon} \geq \frac{k-1}{4\epsilon} \right\} \cup \left\{  \sum_{|x_{j,N} + \epsilon| \geq (R-1)\epsilon} \frac{1}{x_{j,N}+\epsilon} \leq - \frac{k-1}{4\epsilon} \right\}
\end{split}\end{equation}
Indeed, assume that 
\begin{equation}\label{eq:assumeloc} \# [| x_{j,N}'|  < \epsilon ] \geq k~, \quad \# [|x_{j,N}| < (1+\frac4{k-1}) \epsilon ] \leq k~. \end{equation}
 Then we have by interlacing:
\begin{equation}\label{eq:k-1}
\# \left[ x_{j,N}  \in (-\epsilon,\epsilon) \right]= k-1~,
\end{equation}
and on the other hand there are no $x_{j,N}$ at least in one of the intervals
\[ (-(1+\frac4{k-1})\epsilon,  - \epsilon)~, \quad
(+\epsilon, + (1+\frac4{k-1})\epsilon)~; \]
for example, in the second one. Then
\[ \sum_{j=1}^N \frac{1}{x_{j,N}- \epsilon} \geq 0~,\]
whence by  (\ref{eq:k-1}) and (\ref{eq:assumeloc})
\[\sum_{|x_{j,N}-\epsilon| \geq \frac4{k-1}\epsilon} \frac{1}{x_{j,N}-\epsilon} \geq  \sum_{|x_{j,N}| < \epsilon} \frac{- 1}{x_{j,N}-\epsilon}  \geq  \frac{k-1}{2\epsilon}~.\]
Every $x_{j,N}$  contributes at most $(k-1)/(4\epsilon)$ to the sum on the left-hand
side, therefore either $\# [(1+\frac4{k-1})\epsilon  \leq x_{j,N} < R \epsilon ]  \geq 2$
or  
\[ \sum_{|x_{j,N} - \epsilon| \geq (R-1)\epsilon} \frac{1}{x_{j,N}-\epsilon} \geq \frac{k-1}{4\epsilon}~. \]
This proves (\ref{eq:3opt}), and it remains to bound the probability of the two terms
on the right-hand side. By symmetry,
we can focus on the first term, for which we use the following  lemma, proved below (see e.g.\ Breuer--Duits \cite[Theorem 3.1]{BD} for more sophisticated bounds):
\begin{lemma*}For any determinantal process $\mathfrak{D}$ with self-adjoint kernel of
finite rank and any
(bounded Borel measurable) test function $f$,
\[ \mathbb{P} \left\{ \sum_{x \in \mathfrak{D}} f(x) \geq \mathbb E  \sum_{x \in \mathfrak{D}} f(x) + r \right\} \leq \exp(- A^*(r))~, \quad r > 0~,\]
where
\begin{equation}\label{eq:defA}
A^*(r) = \sup_{t \leq\|f\|_\infty^{-1}} (rt - A(t))~, \quad A(t) = \frac{et^2}{2} \, \mathbb{E} \sum_{x \in \mathfrak{D}} f(x)^2~.\end{equation}
\end{lemma*}
\noindent We apply the lemma to
\[ f(x) = \frac{\mathbbm{1}_{|x-\epsilon|\geq (R-1)\epsilon}}{x -\epsilon}~, \quad F_N = \sum f(x_{j,N})~,\]
then, denoting by $\rho_N(E) = \frac1N \frac{d}{dE} \mathbb{E} \# \left\{ \lambda_{j,N} \leq E \right\}$ the mean density of eigenvalues, we have:
\[\begin{split} \mathbb E F_N &= \int_{|x| \geq (R-1)\epsilon} 
 \frac{  \rho_N ( E + \frac{x+\epsilon}{N\rho(E)} )  }   {x} dx \\
&= \int \left[ \frac{\mathbbm{1}_{|x|\geq (R-1)\epsilon}}{x} - \frac{x}{x^2+1} \right] \rho_N ( E + \frac{x+\epsilon}{N\rho(E)} )  dx + \int \frac{x}{x^2+1} \rho_N ( E + \frac{x+\epsilon}{N\rho(E)} ) dx~.\end{split}\]
The first addend tends to zero since $\rho_N \to \rho$ uniformly, whereas the 
second addend tends to $- \frac{\pi E}{\sqrt{4-E^2}}$ by  (\ref{eq:convW}). Therefore
\[\lim_{N \to \infty} \mathbb E F_N = - \frac{\pi E}{\sqrt{4-E^2}}~. \] 
Next, for $t \leq R\epsilon$
\[\begin{split}
\lim_{N \to \infty} A(t) = \frac{et^2}{2} \lim_{N \to \infty} \mathbb{E} \sum  f(x_{j,N})^2
= \frac{et^2}{2} \lim_{N \to \infty} \int_{|x|\geq R\epsilon} \frac{dx}{|x-\epsilon|^2}
\leq \frac{et^2}{(R-1)\epsilon}~,\end{split}\]
whence, taking $t = (R-1)\epsilon$ in the definition (\ref{eq:defA}) of $A^*$ and assuming
that $\epsilon < 1/(16e)$,
\[ \lim_{N \to \infty} A^*\left(\frac{k-1}{8 \epsilon}\right) \geq \frac{k-1}{8\epsilon} (R-1)\epsilon - e (R-1)\epsilon \geq \frac{(k-1)(R-1)}{16} ~. \]
According to the lemma, we have for $\epsilon \leq \min(\frac{\sqrt{4-E^2}}{8\pi|E|},\, \frac{1}{16e})$:
\[ \mathbb{P} \left\{ F_N \geq \frac{k-1}{4 \epsilon} \right\} \leq \exp\left\{ - \frac{(k-1)(R-1)}{16}\right\}  \leq \exp\left\{ - \frac{kR}{64}\right\} ~,\]
as claimed.
\end{proof}

\begin{proof}[Proof of Lemma]
Let $K$ denote  the operator defining the determinantal process; then for
$t \|f\|_\infty \leq 1$
\[\begin{split}
\log \mathbb E \exp\left\{  t \sum_{x \in \mathfrak{D}} f(x)\right\} &= \log \det (1 + (e^{tf}-1)K) = \tr \log   (1 + (e^{tf}-1)K) \\
&\leq \tr (e^{tf}-1)K \leq  \frac{et^2}{2} \tr f^2 K + t \tr fK \\
&\leq A(t) + t \mathbb{E} \sum_{x \in \mathfrak{D}} f(x)~.
\end{split}\]
Therefore by the Chebyshev inequality
\[ \mathbb{P} \left\{ \sum_{x \in \mathfrak{D}} f(x) \geq \mathbb E  \sum_{x \in \mathfrak{D}} f(x) +  r \right\} \leq
\exp(A(t) - rt)~.
\]
\end{proof}
\qedhere

\mypar{pfcor}
\begin{proof}[Proof of Corollary~\ref{cor:limit}']
Let us show that the limit 
\begin{equation}\label{eq:limitck}\mathcal c_k' = \lim_{\epsilon\to 0} \frac{\mathbb{P} (\Omega_k(\mathfrak{Si}'_a, \epsilon))}
{\epsilon^{(k+1)^2}}~, \quad k \geq 2 \end{equation}
exists and does not depend on $a$. Choose $\alpha_k > \alpha_k' > 0$
sufficiently small to ensure that 
\begin{equation}\label{eq:alphak}(1 - \alpha_k)(k+2)^2 > (k+1)^2~.\end{equation}
Denote by $\Upsilon(\epsilon)$ the event
\begin{multline}\label{eq:event}
\exists \quad \text{pairwise distinct distinct} \quad X = (x_1,\cdots,x_{k+1}) \in (\mathfrak{Si} \cap (-\epsilon^{1-\alpha_k},\epsilon^{1-\alpha_k}))^{k+1}~, \\ \text{such that all the zeros of $P_X(x) = \frac{d}{dx} \prod_{j=1}^{k+1} (x - x_j)$ lie 
in $(-\epsilon, \epsilon)$}~. \end{multline}
Let us show that
\begin{equation}\label{eq:1opt} \mathbb P (\Upsilon(\epsilon - \epsilon^{1+\alpha_k'})\setminus\Omega_k(\mathfrak{Si}'_a, \epsilon))~, \quad \mathbb P (\Omega_k(\mathfrak{Si}'_a, \epsilon) \setminus \Upsilon(\epsilon + \epsilon^{1+\alpha_k'})) = o(\epsilon^{(k+1)^2})~. 
\end{equation}
Indeed, on the event $\Upsilon(\epsilon - \epsilon^{1+\alpha_k'}) \setminus \left( \Omega_k(\mathfrak{Si}'_a, \epsilon) \cup \Omega_{k+2}(\mathfrak{Si}, \epsilon^{1-\alpha_k})\right)$ at least one of the $x_j$ lies outside $(-\epsilon, \epsilon)$; assume for example
that $x_{k+1} > \epsilon$ and decompose
\[ 0 > W(\epsilon) =  \sum_{j=1}^{k+1} \frac{1}{x_j - \epsilon} +  \lim_{r \to \infty} \sum_{\substack{ x \in \mathfrak{Si} \\ \epsilon^{1-\alpha_k} \leq |x| \leq r}} \frac{1}{x_j - \epsilon}~.\]
Then the first term is bounded from below by 
\[ \sum_{j=1}^{k+1} \frac{1}{x_j - \epsilon}  \geq \sum_{j=1}^{k+1} \frac{1}{x_j - \epsilon - \epsilon^{1+\alpha_k'} } + \frac{(k+1)\epsilon^{1+\alpha_k'}}{4\epsilon^2} \geq \frac{k+1}{4 \epsilon^{1-\alpha_k'}}~.\]
The probability of the event
\[ \lim_{r \to +\infty} \sum_{\substack{ x \in \mathfrak{Si} \\ \epsilon^{1-\alpha_k} \leq |x| \leq r}} \frac{1}{x_j - \epsilon} \leq -\frac{k+1}{4 \epsilon^{1-\alpha_k'}} \]
is $\mathcal o(\epsilon^{(k+1)^2})$ by a tail estimate which follows from the Lemma in
Section~\ref{rep} (similarly to the 
proof of Theorem~\ref{thm:sine}), whereas $\mathbb{P} (\Omega_{k+2}(\mathfrak{Si}, \epsilon^{1-\alpha_k})) = \mathcal o(\epsilon^{(k+1)^2})$ by the condition (\ref{eq:alphak})
on $\alpha_k$ and (\ref{eq:rep0}). This proves the first part of (\ref{eq:1opt});
the second part is proved in a similar way.

From (\ref{eq:defsine}) and the asymptotics
\[ \det \left(\frac{\sin \pi(x_j - x_m)}{\pi(x_j - x_m)}\right)_{j,m=1}^{k} = 
(1 + \mathcal o(1)) c_{k,1} \prod_{j < m} (x_j - x_m)^2 \quad x \to 0\]
(where $c_{k,1}$ as well as $c_{k,2}$ and $c_{k,3}$ below are numerical constants),
we obtain
\[\begin{split}\mathbb{P}(\Upsilon(\epsilon)) &= (c_{k,2}+ \mathcal o(1)) \epsilon \int_{\sum x_j = 0} d^{k} X   \prod_{j < m} (x_j - x_m)^2 \mathbbm{1}\left\{\text{the zeros of $P_X$ lie in $(-\epsilon, \epsilon)$}\right\} \\
&= (c_{k,2}+ \mathcal o(1)) c_{k,3} \epsilon^{(k+1)^2}~,   \end{split} 
\] 
which implies, with (\ref{eq:1opt}), that (\ref{eq:limitck}) holds with
$\mathcal c_k' = c_{k,2} c_{k,3}$. Note that the indicator under
the integral is compactly supported.

\noindent The bound 
\[ \mathcal c_k' \in [\mathcal c_{k+1}, (1 + \frac{4}{k-1}) \mathcal c_{k+1})^{(k+1)^2}] \]
follows from Theorem~\ref{thm:sine}' and (\ref{eq:rep0}). \qedhere \end{proof}

\mypar{uncond} The proof of Theorem~\ref{thm:xi} relies on a bound, proved
by Rodgers \cite{Rodgers2}, on the moments of
the logarithmic derivative of the $\zeta$-function; see (\ref{eq:rod})
below. The  results of  Farmer, Gonek, Lee and Lester \cite{FGLL} imply
a more precise bound under additional hypotheses.

\begin{proof}[Proof of Theorem~\ref{thm:xi}']
Let $\frac12 + i\gamma_j$ be the zeros of $\xi(z)$; rescale them as follows: $x_{j,t,T} = (\gamma_j - t) \frac{\log T}{2\pi}$. We start with the following counterpart of (\ref{eq:3opt}):
\[\begin{split} 
&\frac1T \operatorname{mes} \left(\Omega_k(\xi', T, \epsilon) \setminus \left(\Omega_{k+1}(\xi, T, (1 + \frac{4}{k-1})\epsilon) \cup \Omega_{k+2}(\xi, T, R\epsilon)\right) \right) \\
&\leq \frac1T \operatorname{mes} \left( \left\{  W_{t,T}(\epsilon) - \sum_{|x_{j,t,T} - \epsilon| < (R-1)\epsilon} \frac{1}{x_{j,t,T}-\epsilon} \geq \frac{k-1}{4\epsilon} \right\} \Big\backslash \Omega_{k+2}(\xi, T, R\epsilon)\right) \\
&+\frac1T \operatorname{mes}  \left( \left\{  W_{t,T}(-\epsilon) - \sum_{|x_{j,t,T} + \epsilon| < (R-1)\epsilon} \frac{1}{x_{j,t,T}+\epsilon} \leq - \frac{k-1}{4\epsilon} \right\} \Big\backslash  \Omega_{k+2}(\xi, T, R\epsilon)\right)
~.\end{split}\]
To show that each of the two terms is bounded by 
$\frac{C}{e^{ckR}}$, it will suffice to prove that
\begin{equation}\label{eq:uncondneed}
\frac1T \operatorname{mes} \left(  \left\{ \left| W_{t,T}(0)- \sum_{|x_{j,t,T}| < R\epsilon} \frac{1}{x_{j,t,T}}\right| \geq \frac{k-1}{4\epsilon} \right\}  \Big\backslash  \Omega_{k+2}(\xi, T, R\epsilon)\right) \leq\frac{C}{e^{ckR}}~. 
\end{equation}
Decompose
\begin{multline*}
W_{t,T}(0)- \sum_{|x_{j,t,T}| < R\epsilon} \frac{1}{x_{j,t,T}} \\
= \Re W_{t,T}(20 iR\epsilon) - \sum_{|x_{j,t,T}| < R\epsilon} \frac{x_{j,t,T}}{x_{j,t,T}^2 + 400 R^2 \epsilon^2}  + \sum_{|x_{j,t,T}| \geq R\epsilon} \frac{400 R^2 \epsilon^2}{x_{j,t,T}(x_{j,t,T}^2 + 400 R^2 \epsilon^2)}\end{multline*}
and observe that on the complement of $\Omega_{k+2}(\xi, T, R \epsilon)$
\begin{equation}\label{eq:oncompl}\left|  \sum_{|x_{j,t,T}| < R\epsilon} \frac{x_{j,t,T}}{x_{j,T}^2 + 400 R^2 \epsilon^2}  \right| \leq 
\frac{k+1}{20 R\epsilon} \leq \frac{k-1}{6 \epsilon}\end{equation}
whereas
\[\left| \Re W_{t,T}(20 iR\epsilon)\right|~, \quad \frac1{20}\left|\sum_{|x_{j,t,T}| \geq R\epsilon} \frac{400 R^2 \epsilon^2}{x_{j,t,T}(x_{j,t,T}^2 + 400 R^2 \epsilon^2)}\right| \leq \left| W_{t,T}(20 iR\epsilon) \right|~,\]
whence we need to show that 
\begin{equation}\label{eq:need}
\frac1T \operatorname{mes} \left\{ 0 \leq t \leq T \, \mid \, |W_{t,T}(20 i R \epsilon)| \geq  \frac{k-1}{480\epsilon}\right\} \leq C \exp(-c kR)~. 
\end{equation}
By a result of Rodgers \cite[Theorem~2.1]{Rodgers2}, 
\begin{equation}\label{eq:rod} \int_0^T \left| \frac{\zeta'(\frac12 + i(t + \frac{2\pi i\delta}{\log T}))}{\zeta(\frac12 + i(t + \frac{2\pi i\delta}{\log T}))}\right|^m dt \leq C^{m \log m} \delta^{-m} \, T \log^m T~,\end{equation}
hence 
\[ \int_0^T \left| \frac{\xi'(\frac12 + i(t + \frac{2\pi i\delta}{\log T}))}{\xi(\frac12 + i(t + \frac{2\pi i\delta}{\log T}))}\right|^m dt \leq C^{m \log m} \delta^{-m} \, T \log^m T~,\]
and, finally,
\begin{equation}\label{eq:logdermoment1} \frac{1}{T} \int_0^T   \left| W_{t,T}(20 iR\epsilon) \right|^m dt \leq C^{m \log m} (R\epsilon)^{-m}~,\end{equation}
from which (\ref{eq:need}) follows by the Chebyshev inequality. 
\end{proof}

\mypar{remarks} We conclude with several questions and comments.

\myspar{r.oth}{Other random matrix ensembles} The results of this paper have
counterparts for other random matrix ensembles (and other $\xi$-functions).
In particular, Chhaibi, Najnudel and Nikeghbali \cite{CNN} showed that the characteristic polynomial $Z_N(z) = \det(\mathbbm{1}_N - z U_N^*)$
of the Circular Unitary Ensemble satisfies:
\[ \frac{Z_N(\exp(2\pi i z/N))}{Z_N(1)} \overset{\text{distr}}\longrightarrow \exp(i \pi z) \Phi(z)~.\]
Hence the collection of critical points of $Z_N(z) z^{-N/2}$ converges, after rescaling
by $N$, to  the same process $\mathfrak{Si}'_0$  as in Corollaries~\ref{aw} and \ref{xi}. 
This is consistent with the discussion in \cite[Sections 2.3 and 6.2]{FGL}.

\myspar{r.comp}{Critical points versus submatrix eigenvalues} Let
\[W^\text{dec} (z) =
\lim_{R \to \infty} \sum\nolimits_{x \in \mathfrak{Si} \cap (-R,R)} \frac{|g_x|^2}{x - z}~,
 \]
 where, conditionally on $\mathfrak {Si}$, the random variables $g_x$ are
 independent complex standard Gaussian. Then Proposition~\ref{aw} has the following 
 counterpart:
\begin{equation}\label{eq:subm}\begin{split}
&\left( \{ \lambda_{j,N} - E) N \rho(E)\}~,
 \{(\lambda_{j,N-1} - E) N \rho(E)\} \right) \\
&\quad\longrightarrow  \big((W^\text{dec})^{-1}(-\frac{\pi E}{\sqrt{4-E^2}}), \, W^{-1}(\infty)\big)~.\end{split}\end{equation}
(Amusingly, the distribution of the first term does not depend on $E$.) We are not 
sure whether (\ref{eq:subm}) has a number-theoretic analogue.

\myspar{r.form}{Form factors} 
Suppose a point process $\mathfrak P$ is a limit point 
of the rescaled zeros of $\xi$, as in Proposition~\ref{xi}. Then the rescaled critical
points converge, along the same subsequence, to the point process $\mathfrak P'  = W_{\mathfrak P}^{-1}(0)$, the distribution of which is uniquely determined by that of
$\mathfrak P$. 
By a result of Montgomery \cite{Mont} and its extension by 
Hejhal \cite{Hej} and Rudnick and Sarnak \cite{RS},  the (multiple) form factors of $\mathfrak P$
coincide, in a restricted domain of momenta, with those of the sine process. 
In view of the research programme suggested by Farmer, Gonek and Lee in
 \cite{FGL}, it is natural to ask which constraints does this impose on the form factors of $\mathfrak P'$, and
to compare these with the results of  \cite{FGL,Bian}.

It would also be interesting to compute the form factor of $\mathfrak{Si}'_a$ and
to check whether it coincides with the right-hand side of (\ref{eq:fgl}) for $a=0$ and $|\alpha| \leq 1$. A possible starting point is the identity 
$e_k(\lambda_1', \cdots, \lambda_{N-1}') = (1 - \frac k N) \, e_k(\lambda_1, \cdots,
\lambda_N)$
relating the elementary symmetric functions in the critical points of a polynomial
to the elementary symmetric functions in the zeros.

\myspar{r.higher}{Zeros of higher derivatives} The result (\ref{eq:fgl}) of \cite{FGL} was extended to higher derivatives of $\xi$ in the Ph.D. thesis of Bian \cite{Bian}. In this context, we
mention that, if the Riemann
hypothesis and the multiple correlation conjecture hold, Corollary~\ref{xi} implies that
for any $k \geq 1$
\[ \sum_{\xi^{(k)}(\frac12 + i \gamma''') =0} \delta(u - (\gamma''' - t) \frac{\log T}{2\pi})
\overset{\text{distr}}\longrightarrow \mathfrak{Si}_0^{(k)} = \{ \Phi^{(k)} = 0\}~. \]

\myspar{r.logder}{Logarithmic derivative}
From Proposition~{\thesubsection} and \cite[Theorems 2.3 and 6.2]{AW}
of Aizenman--Warzel it follows, conditionally on the Riemann hypothesis, that (for $t$ chosen uniformly  in $[0, T]$)
the rescaled logarithmic derivatives
\begin{equation}
\frac{2 i}{\log T} \frac{\xi'(\frac12 + it)}{\xi(\frac12 + it)}~, \quad \frac{2 i}{\log T} \frac{\zeta'(\frac12 + it)}{\zeta(\frac12 + it)} + i \overset{\operatorname{distr}}{\underset{T \to \infty}\longrightarrow} \operatorname{Cauchy}~,\end{equation}
where the right-hand side is a standard (real) Cauchy random variable. For comparison, it was proved by Lester \cite{Lester2}, following earlier results
by Guo \cite{Guo}, that the distribution of $(\zeta'/\zeta)(\sigma(T) + it)$, 
where $|\sigma(T) - \frac12| \log T \to \infty$ and $|\sigma(T) - \frac12| \to 0$, is 
approximately (complex) Gaussian for large $T$.

\bigskip\noindent{\textbf{Acknowledgement}} 
This work originated  from a question of Yan Fyodorov, who asked how to
describe the local statistics of critical points,  and from  subsequent discussions,
particularly, of the strong repulsion; he also made helpful comments and suggestions on all the subsequent stages. Steve Lester patiently 
answered numerous questions and helped me find the way in the 
number-theoretic literature, and, in particular, made the crucial remark that the estimate (\ref{eq:b}) should be true and provable by the methods of \cite{GGM}.  Dave Platt 
numerically computed the saddle points of $\xi(z)$, the spacing statistics of which
are plotted in Figure~\ref{fig:comp}. Jon Keating
and Zeev Rudnick commented on a preliminary version of this paper. I thank them very much.

\end{document}